\documentclass[11pt,reqno]{amsart}
\usepackage{amsmath}
\usepackage{amssymb}
\usepackage[bookmarks]{hyperref}
\newtheorem{theorem}{Theorem}[section]

\newtheorem{proposition}[theorem]{Proposition}

\newtheorem{remark}{Remark}[section]

\newtheorem{problem}[theorem]{Problem}

\numberwithin{equation}{section}

\allowdisplaybreaks[1]
\begin{document}

\title[The zero-product problem]{The zero-product problem for Toeplitz operators with radial symbols}
\author{Trieu Le}
\address{Trieu Le, Department of Mathematics, University of Toronto, Toronto, Ontario, Canada M5S 2E4}
\email{trieu.le@utoronto.ca}
\begin{abstract}
For any bounded measurable function $f$ on the unit ball $\mathbb{B}_n$, let $T_f$ be the Toeplitz operator with symbol $f$ acting on the Bergman space $A^2(\mathbb{B}_n)$. The Zero-Product Problem asks: if $f_1,\ldots, f_N$ are bounded measurable functions such that $T_{f_1}\cdots T_{f_N}=0$, does it follow that one of the functions must be zero almost everywhere? This paper give the affirmative answer to this question when all except possibly one of the symbols are radial functions. The answer in the general case remains unknown.
\end{abstract}

\thanks{}
\date{}
\maketitle

\section{INTRODUCTION}
For $n\geq 1,$ let $\mathbb{C}^{n}$ denote the Cartesian product of $n$ copies of $\mathbb{C}$. For any two points $z=(z_1,\ldots,z_n)$ and $w=(w_1,\ldots,w_n)$ in $\mathbb{C}^n,$ we use the notations $\langle z,w\rangle=z_1\overline{w}_1+\cdots+z_n\overline{w}_n$ and $|z|=\sqrt{|z_1|^2+\cdots+|z_n|^2}$ for the inner product and the associated Euclidean norm. Let $\mathbb{B}_n$ denote the open unit ball which consists of points $z\in\mathbb{C}^n$ with $|z|<1.$ Let ${\rm d}\nu_{n}$ denote the Lebesgue measure on $\mathbb{B}_n$ which is normalized so that $\nu_{n}(\mathbb{B}_n)=1$.

For $1\leq p\leq\infty$ we write $L^{p}=L^{p}(\mathbb{B}_n,\mathrm{d}\nu_{n})$. The Bergman space $A^2$ is the closed subspace of $L^2$ which consists of all functions that are holomorphic in $\mathbb{B}_n$. For any multi-index $m=(m_1,\ldots, m_n)$ in $\mathbb{N}^n$ (throughout the paper $\mathbb{N}$ denotes the set of all non-negative integers), let $|m|=m_1+\cdots+m_n$ and $m!=m_1!\cdots m_n!$. For any $z=(z_1,\ldots, z_n)\in\mathbb{C}^{n}$, let $z^{m}=z_1^{m_1}\cdots z_n^{m_n}$. The standard orthonormal basis for $A^2$ is given by $\{e_{m}: m\in\mathbb{N}^n\}$, where
\begin{equation*}
e_{m}(z) = \Big[\dfrac{(n+|m|)!}{m!\ n!}\Big]^{1/2}z^{m},\ m\in\mathbb{N}^n, z\in\mathbb{B}_n.
\end{equation*}
For a more detailed discussion of $A^2$, see \cite[Chapter 2]{Zhu2005}. Let $P$ denote the orthogonal projection from $L^2$ onto $A^2$. Then we have
\begin{equation*}
(P\varphi)(z)=\int\limits_{\mathbb{B}_n}\dfrac{\varphi(w)}{(1-\langle z,w\rangle )^{n+1}}\ \mathrm{d}\nu_{n}(w),\text{ for }\varphi\in L^2,\ z\in \mathbb{B}_n.
\end{equation*}

For any $f\in L^{2}$, the Toeplitz operator with symbol $f$ is denoted by $T_f$ which is densely defined on $A^2$ by $T_{f}\varphi=P(f\varphi)$ for all bounded holomorphic functions $\varphi$ on $\mathbb{B}_n$. If $f$ is actually a bounded function then $T_{f}$ is a bounded operator on $A^2$ with $\|T_f\|\leq\|f\|_{\infty}$ and $(T_{f})^{*}=T_{\bar{f}}$. If $f$ and $\bar{g}$ are bounded holomorphic functions then $T_{g}T_{f}=T_{gf}$. These properties can be verified directly from the definition of Toeplitz operators.

There is an extensive literature on Toeplitz operators on the Hardy space $H^2$ of the unit circle. We refer to \cite{Martinez-Avendano2007} for definitions of $H^2$ and their Toeplitz operators. In the context of Toeplitz operators on $H^2$, it was showed by Brown and Halmos \cite{Brown1963} back in the sixties that if $f$ and $g$ are bounded functions on the unit circle then $T_{g}T_{f}$ is another Toeplitz operator if and only if either $f$ or $\bar{g}$ is holomorphic. From this it is readily deduced that if $f,g\in L^{\infty}$ such that $T_{g}T_{f}=0$ then one of the symbols must be the zero function. A more general question concerning products of several Toeplitz operators (on $H^2$ or $A^2$) is the so-called ``zero-product'' problem.

\begin{problem}\label{prob-1}
Suppose $f_1,\ldots,f_N$ are functions in $L^{\infty}$ such that $T_{f_1}\cdots T_{f_N}$ is the zero operator. Does it follow that one of the functions $f_{j}$'s must be the zero function?
\end{problem}

For Toeplitz operators on $H^2$, the affirmative answer was proved by K.Y. Guo \cite{Guo1996} for $N=5$ and by C. Gu \cite{Gu2000} for $N=6$. Problem \ref{prob-1} for general $N$ remains open. For Toeplitz operators on the Bergman space of the unit disk, Ahern and {\v{C}}u{\v{c}}kovi{\'c} \cite{Ahern2001} answered Problem \ref{prob-1} affirmatively for the case $N=2$ with an additional assumption that both symbols are bounded harmonic functions. In fact they studied a type of Brown-Halmos Theorem for Toeplitz operators on the Bergman space. They showed that if $f$ and $g$ are bounded harmonic functions and $h$ is a bounded $C^2$ function whose invariant Laplacian is also bounded then the equality $T_{g}T_{f}=T_{h}$ holds only in the trivial case, that is, when $f$ or $\bar{g}$ is holomorphic. As a corollary they obtained their solution to Problem \ref{prob-1}. This result was generalized to Toeplitz operators on the Bergman space of the unit ball in $\mathbb{C}^{n}$ by B. Choe and K. Koo in \cite{Choe2006} with an assumption about the continuity of the symbols on an open subset of the boundary. They were also able to show that if $f_1,\ldots, f_{n+3}$ (here $N=n+3$) are bounded harmonic functions that have Lipschitz continuous extensions to the whole boundary of the unit ball then $T_{f_{1}}\cdots T_{f_{n+3}}=0$ implies that one of the symbols must be zero. The answer in the general case remains unknown, even for two Toeplitz operators.

The purpose of this paper is to report a simple solution to Problem \ref{prob-1} when all but possibly one of the symbols are radial functions (see definition below). Since Toeplitz operators with radial symbols are diagonal, they are easier to work with. But to the best of the author's knowledge, Problem \ref{prob-1} with radial symbols has not been fully studied. The author has been aware of only the paper \cite{Ahern2004} by Ahern and {\v{C}}u{\v{c}}kovi{\'c}, in which, among other things, they showed that in the one dimensional case, if $f$ or $g$ is radial and $T_gT_f=0$ then either $f$ or $g$ is the zero function. It is not clear how their method will work for products of more than two Toeplitz operators or in the setting of Bergman spaces in higher dimensions.

\section{MAIN RESULT}

A function $f$ on $\mathbb{B}_n$ is called a radial function if there is a function $\tilde{f}$ on $[0,1)$ such that $f(z)=\tilde{f}(|z|)$ for all $z\in\mathbb{B}_n$. For any bounded radial function $f$ and any real number $s\geq 0$, put
\begin{equation*}
\omega(f,s)=(n+s)\int\limits_{0}^{1}r^{n+s-1}\tilde{f}(r^{1/2}){\rm d}r.
\end{equation*}
This quantity is related to the Mellin transform which was used in the study of Toeplitz operators with radial symbols in \cite{Ahern2004}. Using integration in polar coordinates, we can show that the operator $T_{f}$ is diagonal with respect to the standard orthonormal basis. In fact, we have
\begin{equation*}
T_{f} = \sum\limits_{m\in\mathbb{N}^n}\omega(f,|m|)e_{m}\otimes e_{m}.
\end{equation*}
Here for $g,h\in A^2$,\ $g\otimes h$ is the operator defined by $(g\otimes h)(\varphi)=\langle\varphi,h\rangle g$ for all $\varphi\in A^2$. The main result of the paper is the following Theorem.

\begin{theorem}\label{theorem-1} Suppose $f_1,\ldots, f_N$ and $g_1,\ldots, g_M$ are bounded radial functions so that none of them is the zero function. Suppose $f$ is a function in $L^2$ such that the operator $T_{f_1}\cdots T_{f_N}T_{f}T_{g_1}\cdots T_{g_M}$ (which is densely defined on $A^2$) is the zero operator. Then $f$ must be the zero function.
\end{theorem}

Since our proof of Theorem \ref{theorem-1} is based on two versions of the {M}\"untz-{S}z\'asz Theorem, we restate them here and refer the interested reader to the references for their proofs. In fact the proof of the second version (which is for continuous functions on the unit disk) relies on the first version (which is for continuous functions on the unit interval).

\begin{theorem}[Theorem $15.26$ in \cite{Rudin1987}]\label{theorem-2} Suppose $0<\lambda_1<\lambda_2<\lambda_3<\cdots$ and let $X$ be the closure in $C([0,1])$ of the set of all finite linear combinations of the functions
\begin{equation*} 1, t^{\lambda_1}, t^{\lambda_2}, t^{\lambda_3}, \ldots.
\end{equation*}
(a) If $\sum_{j=1}^{\infty}\frac{1}{\lambda_j}=\infty$ then $X=C([0,1])$.

\noindent
(b) If $\sum_{j=1}^{\infty}\frac{1}{\lambda_j}<\infty$, and if $\lambda\notin\{\lambda_j\}$, $\lambda\neq 0$, then $X$ does not contain the function $t^{\lambda}$.
\end{theorem}

\begin{theorem}[T. Trent \cite{Trent1981}]\label{theorem-3} Let $M$ be a subset of $\mathbb{N}\times\mathbb{N}$. For any integer $j$, let $M_{j}$ denote the set $\{s\in\mathbb{N}: (s,s+j)\in M\}$. Then the closure in $C(\bar{\mathbb{B}}_1)$ of all finite linear combinations of functions in $\{1, z^{s}\bar{z}^{t}: (s,t)\in M\}$ is $C(\bar{\mathbb{B}}_1)$ if and only if for each integer $j$, $\sum\limits_{s\in M_{j}\backslash\{0\}}\frac{1}{s}=\infty$.
\end{theorem}

The following Proposition is an important step in our proof of Theorem \ref{theorem-1}.
\begin{proposition}\label{prop-1}
Suppose $S$ is a set of non-negative integers such that $\sum\limits_{s\in S\backslash\{0\}}\frac{1}{s}<\infty$. Suppose $f\in L^{1}(\mathbb{B}_n)$ such that $\int\limits_{\mathbb{B}_n}f(z)z^{m}\bar{z}^{k}{\rm d}\nu_{n}(z)=0$ whenever $m,k\in\mathbb{N}^{n}$ with $|m|,|k|\notin S$. Then $f(z)=0$ for almost every $z$ in $\mathbb{B}_n$.
\end{proposition}
\begin{proof}
First suppose $n\geq 2$. For any $z\in\mathbb{B}_n$ we write $z=(z_1,\sqrt{1-|z_1|^2}\,w)$ with $z_1\in\mathbb{B}_1$ and $w\in\mathbb{B}_{n-1}$. Using Fubini's Theorem and the Change of Variables, we then have, for any $g\in L^{1}(\mathbb{B}_n)$,
\begin{equation*}
\int\limits_{\mathbb{B}_n}g(z){\rm d}\nu_{n}(z) = n\int\limits_{\mathbb{B}_1}\!\int\limits_{\mathbb{B}_{n-1}}g(z_1,\sqrt{1-|z_1|^2}\,w)(1-|z_1|^2)^{n-1}{\rm d}\nu_{n-1}(w){\rm d}\nu_{1}(z_1).
\end{equation*}

Let $\tilde{m}$ and $\tilde{k}$ be two multi-indices in $\mathbb{N}^{n-1}$. Let $j\in\mathbb{Z}$ be any integer. Put
\begin{equation*}
M_{j} = \{m_1\in\mathbb{N}^{*}: m_1+|\tilde{m}| \text{ and } m_1+j+|\tilde{k}|\text{ are not in } S\},
\end{equation*}
where $\mathbb{N}^{*}$ denotes $\mathbb{N}\backslash\{0\}$. Then $\mathbb{N}^{*}\backslash M_{j}$ is contained in $(S-|\tilde{m}|)\cup (S-j-|\tilde{k}|)$. Here for any number $a$, $S-a$ denotes the set $\{s-a:s\in S\}$. By our assumption about $S$ we see that $\sum\limits_{m_1\in\mathbb{N}^{*}\backslash M_{j}}\frac{1}{m_{1}}<\infty$. This implies that $\sum\limits_{m_1\in M_j}\frac{1}{m_1}=\infty$.
For any $m_1\in M_j$ and $m_1\geq -j$, let $m=(m_1,\tilde{m})$ and $k=(m_1+j,\tilde{k})$. Then $m$ and $k$ are two multi-indices in $\mathbb{N}^n$ with $|m|,|k|\notin S$. Therefore,
\begin{align*}
0 & = \int\limits_{\mathbb{B}_{n}}f(z)z^{m}\bar{z}^{k}{\rm d}\nu_{n}(z)\\
& = n\int\limits_{\mathbb{B}_1}\int\limits_{\mathbb{B}_{n-1}}f(z_1,\sqrt{1-|z_1|^2}\,w)z_1^{m_1}\bar{z_1}^{m_1+j}w^{\tilde{m}}\bar{w}^{\tilde{k}}\\
& \quad\times (1-|z_1|^2)^{n-1+(|\tilde{m}|+|\tilde{k}|)/2}{\rm d}\nu_{n-1}(w){\rm d}\nu_{1}(z_1)\\
& = n\int\limits_{\mathbb{B}_1}\Big(\int\limits_{B_{n-1}}f(z_1,\sqrt{1-|z_1|^2}\,w)w^{\tilde{m}}\bar{w}^{\tilde{k}}{\rm d}\nu_{n-1}(w)\Big)\\
& \quad\times z_1^{m_1}\bar{z_1}^{m_1+j}(1-|z_1|^2)^{n-1+(|\tilde{m}|+|\tilde{k}|)/2}{\rm d}\nu(z_1).
\end{align*}
Put
\begin{equation*}
f_{\tilde{m},\tilde{k}}(z_1) = \int\limits_{B_{n-1}}f(z_1,\sqrt{1-|z_1|^2}\,w)w^{\tilde{m}}\bar{w}^{\tilde{k}}{\rm d}\nu_{n-1}(w),\quad z_1\in\mathbb{B}_1.
\end{equation*}
 Then for all $j\in\mathbb{Z}$ and $m_1\in M_j$, we have
\begin{equation*}
\int\limits_{\mathbb{B}_1}f_{\tilde{m},\tilde{k}}(z_1)z_1^{m_1}\bar{z_1}^{m_1+j}(1-|z_1|^2)^{n-1+(|\tilde{m}|+|\tilde{k}|)/2}{\rm d}\nu_{1}(z_1) = 0.
\end{equation*}
From Theorem \ref{theorem-3} we conclude that $f_{\tilde{m},\tilde{k}}(z_1)=0$ for almost every $z_1\in\mathbb{B}_1$. Hence there is a null subset $E$ of $\mathbb{B}_1$ such that $f_{\tilde{m},\tilde{k}}(z_1)=0$ for all $z_1\in\mathbb{B}_1\backslash E$ and all $\tilde{m},\tilde{k}\in\mathbb{N}^{n-1}$. So we have
\begin{equation*}
\int\limits_{B_{n-1}}f(z_1,\sqrt{1-|z_1|^2}\,w)w^{\tilde{m}}\bar{w}^{\tilde{k}}{\rm d}\nu_{n-1}(w) = 0
\end{equation*}
for all $\tilde{m},\tilde{k}\in\mathbb{N}^{n-1}$ and all $z\in\mathbb{B}_1\backslash E$. This implies that $f(z)=0$ for almost every $z$ in $\mathbb{B}_n$ since the set of all finite linear combinations of the functions in $\{w^{\tilde{m}}\bar{w}^{\tilde{k}}:\ \tilde{m},\tilde{k}\in\mathbb{N}^n\}$ is dense in $C(\bar{\mathbb{B}}_{n-1})$.

The above argument also works for the case $n=1$ in which there is no $\tilde{m}$ nor $\tilde{k}$.
\end{proof}

We are now ready for a proof of Theorem \ref{theorem-1}.

\begin{proof}
We have seen that if $g$ is a bounded radial function then $T_g$ is a bounded diagonal operator and we have
\begin{equation*}
T_{g} = \sum\limits_{m\in\mathbb{N}^n}\omega(g,|m|)e_{m}\otimes e_{m}.
\end{equation*}
For such a $g$, put $Z(g)=\{m\in\mathbb{N}^{n}: T_{g}e_{m}=0\}=\{m\in\mathbb{N}^{n}: \omega(g,|m|)=0\}$. It is then clear that if $|m|=|k|$ and $k\in Z(g)$ then $m\in Z(g)$. Put $W(g)=\{|m|: m\in Z(g)\}$. It follows from Theorem \ref{theorem-2} that if $g$ is not the zero function, then $\sum\limits_{s\in W(g)\backslash\{0\}}\frac{1}{s}<\infty$.

Now let $f_1,\ldots, f_N$ and $g_1,\ldots, g_M$ be as in the hypothesis of Theorem \ref{theorem-1}. We define
\begin{align*}
Z & = Z(\bar{f}_1)\cup\cdots\cup Z(\bar{f}_N)\cup Z(g_1)\cup\cdots\cup Z(g_M),\text{and}\\
W & = W(\bar{f}_1)\cup\cdots\cup W(\bar{f}_N)\cup W(g_1)\cup\cdots\cup W(g_M).
\end{align*}
Since none of the functions $f_1,\ldots,f_N,g_1,\ldots,g_M$ is the zero function, we have $\sum\limits_{s\in W\backslash\{0\}}\frac{1}{s}<\infty$. For any multi-indices $m,k\in\mathbb{N}^{n}$ with $|m|,|k|\notin W$ we have $m,k\notin Z$. Hence $T_{g_1}\cdots T_{g_M}e_{m}=c_{m}e_{m}$ and $T_{\bar{f}_N}\cdots T_{\bar{f}_1}e_{k}=d_{k}e_{k}$ with $c_{m},d_{k}\neq 0$. Therefore,
\begin{align*}
\langle fe_{m},e_{k}\rangle & = \langle T_{f}e_{m}, e_{k}\rangle\\
& = \dfrac{1}{c_{m}\bar{d}_{k}}\langle T_{f}T_{g_1}\cdots T_{g_M}e_{m}, T_{\bar{f}_N}\cdots T_{\bar{f}_1}e_{k}\rangle\\
& = \dfrac{1}{c_{m}\bar{d}_{k}}\langle T_{f_1}\cdots T_{f_N} T_{f}T_{1}\cdots T_{g_M} e_{m}, e_{k}\rangle = 0.
\end{align*}
This implies that $\int\limits_{\mathbb{B}_{n}} f(z)z^{m}\bar{z}^{k}{\rm d}\nu_{n}(z)=0$ whenever $m,k\in\mathbb{N}^n$ such that $|m|,|k|\notin W$. Proposition \ref{prop-1} now shows that $f(z)=0$ for almost every $z\in\mathbb{B}_n$.
\end{proof}

\begin{remark} We would like to remark here that with essentially the same proof, Theorem \ref{theorem-1} also holds for Toeplitz operators on weighted Bergman spaces $A^2_{\alpha}(\mathbb{B}_n)$ with $\alpha>-1$.
\end{remark}

\noindent
\textit{Acknowledgments} This work was completed when the author was attending the Fall 2007 Thematic Program on Operator Algebras at the Fields Institute for Research in Mathematical Sciences.

\end{document}